\documentclass[runningheads]{llncs}

\usepackage{graphicx}
\usepackage{amsmath}
\usepackage{amsfonts}

\DeclareMathOperator*{\argmin}{arg\!\min}
\DeclareMathOperator*{\argmax}{arg\!\max}

\usepackage{algorithm} % para escribir pseudocodigo
\usepackage{algorithmic}
\begin{document}
\title{Re-Weighted $\ell_{1}$ Algorithms within the Lagrange Duality Framework}
\subtitle{Bringing Interpretability to Weights}

%\titlerunning{Abbreviated paper title}
% If the paper title is too long for the running head, you can set
% an abbreviated paper title here
%
\author{Mat\'ias Vald\'es\inst{1} \and Marcelo Fiori\inst{1}}
\authorrunning{M. Vald\'es \and M. Fiori}
% First names are abbreviated in the running head.
% If there are more than two authors, 'et al.' is used.
%
\institute{Instituto de Matem\'atica y Estad\'istica Rafael Laguardia\\Facultad de Ingenier\'ia, Universidad de la Rep\'ublica, Uruguay\\
\email{\{mvaldes,mfiori\}@fing.edu.uy}}
\maketitle              % typeset the header of the contribution
\begin{abstract} % (150-250 words)

We consider an important problem in signal processing, which consists in finding the sparsest solution of a linear system $\Phi x=b$. This problem has applications in several areas, but is NP-hard in general. Usually an alternative convex problem is considered, based on minimizing the (weighted) $\ell_{1}$ norm. For this alternative to be useful, weights should be chosen as to obtain a solution of the original NP-hard problem. A well known algorithm for this is the Re-Weighted $\ell_{1}$, proposed by Cand\`es, Wakin and Boyd. In this article we introduce a new methodology for updating the weights of a Re-Weighted $\ell_{1}$ algorithm, based on identifying these weights as Lagrange multipliers. This is then translated into an algorithm with performance comparable to the usual methodology, but allowing an interpretation of the weights as Lagrange multipliers. The methodology may also be used for a noisy linear system, obtaining in this case a Re-Weighted LASSO algorithm, with a promising performance according to the experimental results.

\end{abstract}

\section{Introduction}

An important problem in signal processing, particularly in the field of compressed sensing and sparse coding, is to find the ``sparsest'' solution of a linear system $\Phi x =b$; being $\Phi \in \mathbb{R}^{m \times n}$, $m<n$. That is: a solution with as many null coordinates as possible. This problem has applications in several areas like  \cite{qaisar_2013}: medical imaging, error correcting, digital cameras and wireless communication. Sparsity may be measured by the $\ell_{0}$ pseudo-norm $||x||_{0}$, which counts the number of non zero coordinates. The problem of interest is then: \begin{equation} \label{eq:compressed_sensing} \tag{$P_{0}$}
\argmin_{
\begin{array}{ll}
	\Phi x=b \\
	x \in \mathbb{R}^{n}
\end{array}
} \:\: ||x||_{0}.
\end{equation}

This problem is NP-hard in general \cite{foucart_2013}. A usual alternative is to replace the $\ell_{0}$ pseudo-norm by a weighted $\ell_{1}$ norm: \begin{equation} \label{eq:l1_pesos} \tag{$P_{1}W$}
x^{w} \in \argmin_{
\begin{array}{ll}
	\Phi x=b \\
	x \in \mathbb{R}^{n}
\end{array}
} \:\: \sum \limits_{i=1}^{n} w_{i} |x_{i}|.
\end{equation}

Problem \eqref{eq:l1_pesos} is convex and so it may be solved efficiently, although it is not always equivalent to \eqref{eq:compressed_sensing}. Note that $\ell_{1}$ minimization is obtained by using unit weights in \eqref{eq:l1_pesos}. For this particular case there are important results about its equivalence with \eqref{eq:compressed_sensing}, mainly due to Donoho \cite{donoho_2006} and Cand\`es, Romberg and Tao \cite{candes_2004}. For the general case, the task is to choose ``useful weights'' for \eqref{eq:l1_pesos}, defined as those that make $x^{w}$ be a solution of \eqref{eq:compressed_sensing}. Cand\`es, Wakin and Boyd (CWB) proposed an iterative algorithm, known as Re-Weighted $\ell_{1}$ (RW$\ell_{1}$), to estimate useful weights \cite{candes_2008}. The algorithm updates weights as follows: \begin{equation}
w_{i}^{k+1} = \frac{1}{|x_{i}^{k}|+ \epsilon_{k}}, \: \forall \: k \geq 0;
\end{equation} for some $\epsilon_{k}>0$ and with: \begin{equation} \label{eq:pesos_cwb}
x^{k} \in \argmin_{
\begin{array}{ll}
	\Phi x=b \\
	x \in \mathbb{R}^{n}
\end{array}
} \: \sum \limits_{i=1}^{n} w_{i}^{k} |x_{i}|, \: \forall \: k \geq 0.
\end{equation}

In this work we propose a new methodology to estimate weights, based on the theory of Lagrange duality. Using this methodology, together with an algorithm for estimating solutions from a dual problem, we obtain a new RW$\ell_{1}$ algorithm. The methodology is also applied to a noisy linear system, obtaining in this case a Re-Weighted LASSO algorithm (RW-LASSO).

The rest of the paper is organized as follows: Section \ref{sec:oracle} introduces the proposed methodology in the oracle case, in which a solution of \eqref{eq:compressed_sensing} is known. Here an oracle dual problem is obtained. Section \ref{sec:dual_soluciones} describes some solutions of this dual problem. In Section \ref{sec:rwl1_subgradiente} a new RW$\ell_{1}$ algorithm is obtained by applying the proposed methodology with the subgradient algorithm. Section \ref{sec:oracle_sin} extends the methodology and the RW$\ell_{1}$ subgradient algorithm to the non-oracle case, in which no solution of \eqref{eq:compressed_sensing} is known. Section \ref{sec:ruido} generalizes the methodology for the case in which the linear system is affected by noise. Here a RW-LASSO algorithm is obtained. Section \ref{sec:resultados} analices the performance of the proposed RW$\ell_{1}$ algorithm in the noiseless case, and the RW-LASSO algorithm in the noisy case, both applied to random linear systems. Section \ref{sec:conclusiones} gives the final conclusions.

\section{Methodology with Oracle} \label{sec:oracle}

The proposed methodology is introduced in the ideal case in which a solution $x^{*}$ of \eqref{eq:compressed_sensing} is known. Consider the ideal primal problem defined as: \begin{equation} \label{eq:primal_ideal} \tag{$P$}
\argmin_{
\begin{array}{ll}
	\Phi x=b \\
	|x_{i}| \leq |x^{*}_{i}|,\: \forall i
\end{array}
} \:\ 0.
\end{equation} This is a convex problem, so it can be solved efficiently. Also, any solution of \eqref{eq:primal_ideal} is a solution of \eqref{eq:compressed_sensing}. Of course \eqref{eq:primal_ideal} is ideal, since $x^{*}$ is assumed to be known, so it has no practical value. Consider the Lagrange relaxation obtained by relaxing only the constraints involving $x^{*}$. The associated Lagrangian is: \begin{equation}
L(x,w) = \sum \limits_{i=1}^{n} w_{i} \left( |x_{i}| - |x^{*}_{i}| \right) = \sum \limits_{i=1}^{n} w_{i} |x_{i}| - \sum \limits_{i=1}^{n} w_{i} |x^{*}_{i}|,
\end{equation} where $w_{i} \geq 0$ are the Lagrange multipliers. The dual function is then: \begin{equation}
d(w) := \min_{
\begin{array}{ll}
	\Phi x=b \\
	x \in \mathbb{R}^{n}
\end{array}
} \:\: L \left( w, x \right) = \left( \min_{
\begin{array}{ll}
	\Phi x=b \\
	x \in \mathbb{R}^{n}
\end{array}
} \:\: \sum \limits_{i=1}^{n} w_{i} |x_{i}| \right) - \sum \limits_{i=1}^{n} w_{i} |x^{*}_{i}|.
\end{equation} This dual function involves a Weighted $\ell_{1}$ problem, in which weights are Lagrange multipliers. This is the key idea behind the proposed methodology: identify weights of \eqref{eq:l1_pesos} as Lagrange multipliers. The problem is then in the context of Lagrange duality. In particular, weights may be estimated by any algorithm to estimate multipliers. Equivalently, weights may be estimated as solutions of the dual problem, given by: \begin{equation} \label{eq:dual_ideal_relajado} \tag{$D$}
\argmax_{
\begin{array}{ll}
	w \geq 0 \\
	w \in \mathbb{R}^{n}
\end{array}
} \: d(w).
\end{equation} This maximization problem is always concave, so it may be solved efficiently. One drawback is that usually the dual function is non differentiable, so for example gradient based algorithms must be replaced by subgradient.

\section{Solutions of the Dual Problem} \label{sec:dual_soluciones}

Now the interest is to find ``useful solutions'' of the dual \eqref{eq:dual_ideal_relajado}. That is: $w \geq 0$ such that $x^{w}$ is a solution of \eqref{eq:compressed_sensing}. This section shows that such solutions always exist, although not every solution of \eqref{eq:dual_ideal_relajado} has this property.

\begin{proposition}
Primal problem \eqref{eq:primal_ideal} satisfies strong duality: $d^{*}=f^{*}$.
\end{proposition}

\begin{proof}
The primal optimal value is clearly $f^{*}=0$. By weak duality: $d^{*} \leq f^{*}=0$. So it suffices to show that $d(w)=0$, for some $w \geq 0$. Taking $w=\vec{0}$: \begin{equation}
d(\vec{0}) = \left( \min_{
\begin{array}{ll}
	\Phi x=b \\
	x \in \mathbb{R}^{n}
\end{array}
} \:\: \sum \limits_{i=1}^{n} 0 |x_{i}| \right) - \sum \limits_{i=1}^{n} 0 |x^{*}_{i}|=0.
\end{equation}
\qed
\end{proof}

It was also shown that $w=\vec{0}$ is a solution of \eqref{eq:dual_ideal_relajado}. Clearly $w = \vec{0}$ is not necessarily a useful solution, since $x^{\vec{0}}$ could be any solution of the linear system: $$x^{\vec{0}} \in \argmin_{ \begin{array}{ll}
	\Phi x=b \\
	x \in \mathbb{R}^{n}
\end{array}
} \:\: \sum \limits_{i=1}^{n} 0 |x_{i}| = \{ \Phi x=b \}.$$

A consequence of strong duality is that the set of Lagrange multipliers and of dual solutions are equal. Therefore, useful weights may be estimated as dual solutions. The following result shows that the dual problem always admits useful weights as solutions.

\begin{proposition}
Let $\hat{w} \geq 0$ such that: $\hat{w}_{i}=0 \Leftrightarrow x_{i}^{*} \neq 0$. Then every solution $x^{\hat{w}}$ of the problem \eqref{eq:l1_pesos} associated to $\hat{w}$, is a solution of \eqref{eq:compressed_sensing}.
\end{proposition}

\begin{proof}
Let $I = \{ i \: / \: x_{i}^{*}=0 \}$. By definition of $\hat{w}$ and $x^{\hat{w}}$, and using that $\Phi x^{*} = b$: \begin{equation} 0 \leq \sum \limits_{i \in I}^{} \hat{w}_{i}|x^{\hat{w}}_{i}| = \sum \limits_{i=1}^{n} \hat{w}_{i}|x^{\hat{w}}_{i}| \leq \sum \limits_{i=1}^{n} \hat{w}_{i}|x^{*}_{i}| =  \sum \limits_{i \in I}^{} \hat{w}_{i}|x^{*}_{i}| = 0. \end{equation} This implies: $\hat{w}_{i}|x^{\hat{w}}_{i}| = 0, \: \forall \: i \in I$. Since $\hat{w}_{i} > 0, \: \forall \: i \in I$, then we must have: $x^{\hat{w}}_{i} = 0, \: \forall \: i \in I$. So: $||x^{\hat{w}}||_{0} \leq ||x^{*}||_{0}$. By definition $\Phi x^{\hat{w}}=b$, then it solves \eqref{eq:compressed_sensing}.
\qed
\end{proof}

\section{RW$\ell_{1}$ with Projected Subgradient Algorithm} \label{sec:rwl1_subgradiente}

In this section we give an implementation of the proposed methodology, by using the projected subgradient algorithm for estimating solutions of the dual problem. This algorithm may be thought as a (sub)gradient ``ascent'', with a projection on the dual feasible set. More specifically, starting at $w^{0} \geq 0$, the update is: \begin{equation}
\left\lbrace \begin{array}{ll}
w^{k+1} = w^{k} + \alpha_{k} g^{k}\\
w^{k+1} = \max \{ 0,w^{k+1} \} 
\end{array} \right., \forall \: k \geq 0;
\end{equation} where $g^{k} \in \partial d(w^{k})$ is a subgradient of the dual function at $w^{k}$, and $\alpha_{k} >0$ the stepsize. Although this is not strictly an ascent method, it is always possible to choose the stepsize in order to decrease the distance of $w^{k}$ to the dual solution set. A way for this is to update the stepsize as \cite{bertsekas_1999}: \begin{equation}
\alpha_{k} = \frac{d^{*}-d(w^{k})}{||g^{k}||^{2}_{2}} \geq 0, \; \forall \: k \geq 0.
\end{equation}

Applying \cite{bertsekas_2015}[Example 3.1.2] to \eqref{eq:primal_ideal}, it can be seen that a subgradient $g^{k} \in \partial d(w^{k})$ can be obtained by solving a Weighted $\ell_{1}$ problem: \begin{equation}
x^{k} \in \argmin_{
\begin{array}{ll}
	\Phi x=b \\
	x \in \mathbb{R}^{n}
\end{array}
} \:\: \sum \limits_{i=1}^{n} w_{i}^{k} |x_{i}| \Rightarrow g(x^{k}) \in \partial d(w^{k}), \: \forall \: k \geq 0.
\end{equation}

Note that the stepsize can now be written as: \begin{equation} \label{eq:paso_optimo_ideal}
\alpha_{k} = \frac{d^{*}-d(w^{k})}{||g^{k}||^{2}_{2}} = \frac{0-L(x^{k},w^{k})}{||g(x^{k})||_{2}^{2}} = - \frac{ \sum \limits_{i=1}^{n} w_{i}^{k} \left( |x_{i}^{k}| - |x^{*}| \right) }{\sum \limits_{i=1}^{n} \left( |x_{i}^{k}| - |x^{*}| \right)^{2}}, \: \forall \: k \geq 0.
\end{equation}

Algorithm \ref{alg:dual_ideal} shows a pseudocode of the proposed RW$\ell_{1}$ subgradient algorithm. \begin{algorithm}[H] % algorithm environment
\caption{RW$\ell_{1}$ with projected subgradient (with oracle and noise-free)} \label{alg:dual_ideal}
\begin{algorithmic}[1] % enter the algorithmic environment
    \REQUIRE $\Phi \in \mathbb{R}^{m \times n}$, $b \in \mathbb{R}^{m}$, $w^{0} \geq 0$, $\text{RWIter} \geq 0$
	\STATE $x^{0} \in \argmin_{
\begin{array}{ll}
	\Phi x=b \\
	x \in \mathbb{R}^{n}
\end{array}
} \:\: \sum \limits_{i=1}^{n} w^{0}_{i}|x_{i}|$ $\:\:$ \COMMENT{\eqref{eq:l1_pesos}}
	\STATE
    \STATE $k=0$
    \WHILE{$k < \text{RWIter}$}
    \STATE $g_{i}^{k}=g_{i}(x^{k})=|x_{i}^{k}|-|x_{i}^{*}|$ \COMMENT{subgradient at $w^{k}$}
    \STATE
	\STATE Choose $\alpha_{k}$ using \eqref{eq:paso_optimo_ideal}
	\STATE
	\STATE $w_{i}^{k+1} = w_{i}^{k} + \alpha_{k} g_{i}^{k}$
	\STATE $w_{i}^{k+1} = \max \left( 0, w_{i}^{k+1} \right)$
	\STATE
	\STATE $x^{k+1} \in \argmin_{
\begin{array}{ll}
	\Phi x=b \\
	x \in \mathbb{R}^{n}
\end{array}
} \:\: \sum \limits_{i=1}^{n} w_{i}^{k+1} |x_{i}|$ $\:\:$ \COMMENT{\eqref{eq:l1_pesos} with warm restart $x^{k}$}
	\STATE
	\STATE $k = k+1$
	\STATE
    \ENDWHILE
    \RETURN $x^{k}$
\end{algorithmic}
\end{algorithm}

\section{Methodology and Algorithm without Oracle} \label{sec:oracle_sin}

The proposed methodology is now extended to the practical case, in which no solution of \eqref{eq:compressed_sensing} is known. A simple way for doing this is to replace $x^{*}$ in the ideal constraints by its best known estimate $x^{k}$, ``amplified'' by some $\epsilon_{k}>0$: \begin{equation}
g_{i}^{k}(x)= |x_{i}| - \left( 1 + \epsilon_{k} \right) |x_{i}^{k}|, \: \forall \: k \geq 0;
\end{equation} where $x^{k}$ is calculated in the same way as in the oracle case: $$x^{k} \in \argmin_{
\begin{array}{ll}
	\Phi x=b \\
	x \in \mathbb{R}^{n}
\end{array}
} \:\: \sum \limits_{i=1}^{n} w_{i}^{k} |x_{i}|, \: \forall \: k \geq 0.$$

This gives specific constraints $g^{k}(\cdot)$ for each step $k$, and their respective primal problem: \begin{equation} \label{eq:primal_no_ideal} \tag{$P^{k}$}
\argmin_{
\begin{array}{ll}
	\Phi x=b \\
	|x_{i}| \leq \left( 1 + \epsilon_{k} \right) |x_{i}^{k}|, \: \forall i \\
	x \in \mathbb{R}^{n}
\end{array}
} \:\: 0.
\end{equation}

Since $x^{k}$ is always feasible at \eqref{eq:primal_no_ideal}, this problem has optimal value $f^{k}=0$. By relaxing its non-ideal constraints, a dual problem may be obtained. The Lagrange an dual functions are, respectively: \begin{equation}
L^{k}(x,w) = \sum \limits_{i=1}^{n} w_{i} |x_{i}| - \sum \limits_{i=1}^{n} w_{i} \left( 1 + \epsilon_{k} \right) |x^{k}_{i}|,
\end{equation} \begin{equation}
d^{k}(w) = \left( \min_{
\begin{array}{ll}
	\Phi x=b \\
	x \in \mathbb{R}^{n}
\end{array}
} \:\: \sum \limits_{i=1}^{n} w_{i} |x_{i}| \right) - \sum \limits_{i=1}^{n} w_{i} \left( 1 + \epsilon_{k} \right) |x^{k}_{i}|.
\end{equation} Like in the oracle case, each dual function involves a Weighted $\ell_{1}$ problem, with weights as Lagrange multipliers. This allows to extend the methodology, by estimating weights of \eqref{eq:primal_no_ideal} as Lagrange multipliers, or solving its dual problem: \begin{equation} \label{eq:dual_no_ideal} \tag{$D^{k}$}
\argmax_{
\begin{array}{ll}
	w \geq 0 \\
	w \in \mathbb{R}^{n}
\end{array}
} \: d^{k}(w).
\end{equation}

Solutions of \eqref{eq:dual_no_ideal} may be analized in a similar way as for \eqref{eq:dual_ideal_relajado}. In particular, it can be easily seen that \eqref{eq:primal_no_ideal} satisfies strong duality, with optimal values $f^{k}=d^{k}=0$. It is very useful to know the optimal value $d^{k}$ for \eqref{eq:dual_no_ideal}, in order to compute the stepsize for the subgradient algorithm, when applied to \eqref{eq:dual_no_ideal}: \begin{equation} \label{eq:paso_optimo_noideal}
\alpha_{k} = \frac{d^{k}-d^{k}(w^{k})}{||g^{k}(x^{k})||_{2}^{2}} = \frac{0-L^{k}(x^{k},w^{k})}{||g^{k}(x^{k})||_{2}^{2}} = \frac{1}{\epsilon_{k}} \frac{ \|W^{k} x^{k} \|_{1} }{\|x^{k}\|_{2}^{2}} \geq 0.
\end{equation} Algorithm \ref{alg:dual_no_ideal} shows the pseudo-code of the non-oracle RW$\ell_{1}$ method, obtained by combining the proposed methodology with the projected subgradient algorithm.

\begin{algorithm}[H] % algorithm environment
\caption{RW$\ell_{1}$ with projected subgradient (without oracle and noise free)} \label{alg:dual_no_ideal}
\begin{algorithmic}[1] % enter the algorithmic environment
    \REQUIRE $\Phi \in \mathbb{R}^{m \times n}$, $b \in \mathbb{R}^{m}$, $w^{0} \geq 0$, $\text{RWIter} \geq 0$
    \STATE
	\STATE $x^{0} \in \argmin_{
\begin{array}{ll}
	\Phi x=b \\
	x \in \mathbb{R}^{n}
\end{array}
} \:\: \sum \limits_{i=1}^{n} w^{0}_{i}|x_{i}|$ $\:\:$ \COMMENT{\eqref{eq:l1_pesos}}
	\STATE
    \STATE $k=0$
    \WHILE{$k < \text{RWIter}$}
    \STATE $g_{i}^{k}=g_{i}(x^{k})=|x_{i}^{k}|-\left( 1 + \epsilon_{k} \right) |x_{i}^{k}| = - \epsilon_{k} |x_{i}^{k}|$ \COMMENT{subgradient of $d^{k}$ at $w^{k}$}
    \STATE
	\STATE Choose $\alpha_{k}$ using \eqref{eq:paso_optimo_noideal}
	\STATE
	\STATE $w_{i}^{k+1} = \max \left( 0, w_{i}^{k} + \alpha_{k} g_{i}^{k} \right)$
	\STATE
	\STATE $x^{k+1} \in \argmin_{
\begin{array}{ll}
	\Phi x=b \\
	x \in \mathbb{R}^{n}
\end{array}
} \:\: \sum \limits_{i=1}^{n} w_{i}^{k+1} |x_{i}|$ $\:\:$ \COMMENT{\eqref{eq:l1_pesos} with warm restart $x^{k}$}
	\STATE
	\STATE $k = k+1$
	\STATE
    \ENDWHILE
    \RETURN $x^{k}$
\end{algorithmic}
\end{algorithm}

At each step of Algorithm \ref{alg:dual_no_ideal}, and before the projection, the update is: $$w_{i}^{k+1} = w_{i}^{k} + \alpha_{k} g_{i}^{k}(x^{k}) = w_{i}^{k} - \frac{ \|W^{k} x^{k} \|_{1} }{\|x^{k}\|_{2}^{2}} |x_{i}^{k}|, \: \forall \: k \geq 0;$$ so Algorithm \ref{alg:dual_no_ideal} is independent of $\epsilon_{k}>0$. We take $\epsilon_{k}=1, \: \forall \: k \geq 0$.

\section{Problem with Noise} \label{sec:ruido}

In this section we consider the case in which the linear system is affected by noise. That is: $b = \Phi x^{*} + z$, where $z$ represents the noise. The problem of interest is now: \begin{equation} \label{eq:compressed_sensing_ruido} \tag{$P_{0}^{\eta}$}
\argmin_{
\begin{array}{ll}
	\frac{1}{2} ||\Phi x-b||^{2}_{2} \leq \frac{\eta^{2}}{2}\\
	x \in \mathbb{R}^{n}
\end{array}
} \: ||x||_{0}.
\end{equation} This problem is also NP-hard in general, for any level of noise $\eta \geq 0$ \cite{foucart_2013}. Replacing the $\ell_{0}$ pseudo-norm by a weighted $\ell_{1}$ norm, we obtain a convex alternative: \begin{equation} \label{eq:l1_pesos_ruido} \tag{$P_{1}^{\eta}W$}
\argmin_{
\begin{array}{ll}
	\frac{1}{2} ||\Phi x-b||^{2}_{2} \leq \frac{\eta^{2}}{2}
\end{array}
} \: \sum \limits_{i=1}^{n} w_{i} |x_{i}|.
\end{equation} The proposed methodology is the same as in the noiseless case. Now the oracle primal problem is:
\begin{equation}
\argmin_{
\begin{array}{ll}
	\frac{1}{2} ||\Phi x-b||^{2}_{2} \leq \frac{\eta^{2}}{2}\\
	|x_{i}| \leq |x^{*}_{i}|, \: \forall \: i
\end{array}
} \: 0.
\end{equation} The Lagrangian obtained by relaxing the ideal constraints is the same as for the noiseless case. The dual function is now: \begin{equation}
d(w) = \left( \min_{
\begin{array}{ll}
	\frac{1}{2} ||\Phi x-b||^{2}_{2} \leq \frac{\eta^{2}}{2}
\end{array}
} \: \sum \limits_{i=1}^{n} w_{i} |x_{i}| \right) - \sum \limits_{i=1}^{n} w_{i} |x_{i}^{*}|.
\end{equation} This is a Weighted $\ell_{1}$ problem with quadratic constraints. Such as in the noiseless case, weights can be identified with Lagrange multipliers. So the methodology and the RW$\ell_{1}$ subgradient algorithm are the same as for the noiseless case, but replacing \eqref{eq:l1_pesos} with \eqref{eq:l1_pesos_ruido}. Going a step further, if the quadratic constraints are also relaxed, a new dual function may be obtained: \begin{equation}
d(w,\lambda) = \left( \min_{
\begin{array}{ll}
	x \in \mathbb{R}^{n}
\end{array}
} \: \frac{\lambda}{2}\|\Phi x-b\|_{2}^{2} + \sum \limits_{i=1}^{n} w_{i} |x_{i}| \right) - \left( \frac{\lambda}{2} \eta^{2} + \sum \limits_{i=1}^{n} w_{i} |x_{i}^{*}| \right).
\end{equation} This involves the well known Weighted LASSO problem, which is a simple generalization of the LASSO problem, introduced by Tibshirani in the area of statistics \cite{tibshirani_1996}. Chen, Donoho and Saunders introduced the same LASSO problem in the context of signal representation, but with the name of Basis Pursuit Denoising \cite{chen_2001}. Note that useful weights of \eqref{eq:l1_pesos_ruido} can still be estimated as part of the Lagrange multipliers; which are now $w \in \mathbb{R}^{n}_{+}$ and $\lambda \in \mathbb{R}_{+}$. When combined with the projected subgradient algorithm, this gives a RW-LASSO algorithm, in which at each step a Weighted-LASSO problem must be solved instead of \eqref{eq:l1_pesos_ruido}: \begin{equation}
x^{k} \in \argmin_{x \in \mathbb{R}^{n}} \:\: \frac{\lambda^{k}}{2}||\Phi x-b||^{2}_{2} + \sum \limits_{i=1}^{n} w_{i}^{k} |x_{i}|, \: \forall \: k \geq 0.
\end{equation}

Algorithm \ref{alg:dual_no_ideal_ruido} shows a pseudocode for the proposed subgradient RW-LASSO algorithm.

\begin{algorithm}[H] % algorithm environment
\caption{RW-LASSO with subgradient (without oracle and with noise)}
\label{alg:dual_no_ideal_ruido}
\begin{algorithmic}[1] % enter the algorithmic environment
    \REQUIRE $\Phi \in \mathbb{R}^{m \times n}$, $b \in \mathbb{R}^{m}$, $w^{0} \geq 0$, $\lambda^{0} \in \mathbb{R}$, $\eta \geq 0$, $\text{RWIter} \geq 0$
	\STATE $x^{0} \in \argmin_{x \in \mathbb{R}^{n}} \:\: \frac{\lambda^{0}}{2} \|\Phi x-b\|^{2}_{2} + \sum \limits_{i=1}^{n} w_{i}^{0} |x_{i}|$ \COMMENT{Weighted-LASSO}
	\STATE
    \STATE $k=0$
    \WHILE{$k < \text{RWIter}$}
    \STATE $g_{i}^{k}=g_{i}^{k}(x^{k})=|x_{i}^{k}|-\left( 1 + \epsilon_{k} \right) |x_{i}^{k}|$
	\STATE $w_{i}^{k+1} = \max \left( 0, w_{i}^{k} + \alpha_{k} g_{i}^{k} \right)$
    \STATE
    \STATE $g^{k}_{\lambda} = g_{\lambda}(x^{k})=\frac{1}{2} \left( \|\Phi x^{k}-b\|^{2}_{2} - \eta^{2} \right)$
	\STATE $\lambda^{k+1} = \max \left( 0, \lambda^{k} + \alpha_{k} g^{k}_{\lambda} \right)$
	\STATE
	\STATE $x^{k+1} \in \argmin_{x \in \mathbb{R}^{n}} \:\: \frac{\lambda^{k+1}}{2}||\Phi x-b||^{2}_{2} + \sum \limits_{i=1}^{n} w_{i}^{k+1} |x_{i}|$ \COMMENT{with warm restart $x^{k}$}
	\STATE
	\STATE $k = k+1$
	\STATE
    \ENDWHILE
    \RETURN $x^{k}$
\end{algorithmic}
\end{algorithm}

CWB RW$\ell_{1}$ algorithm can also be extended to the noisy model, by updating weights as in the noiseless case, but taking \cite{candes_2008}: \begin{equation}
x^{k} \in \argmin_{
\begin{array}{ll}
	\frac{1}{2} ||\Phi x-b||^{2}_{2} \leq \frac{\eta^{2}}{2}
\end{array}
} \: \sum \limits_{i=1}^{n} w_{i}^{k} |x_{i}|, \: \forall \: k \geq 0.
\end{equation}

\section{Experimental Results} \label{sec:resultados}

\subsection{Results for the Noise-free Setting}

This section analyzes the performance of the proposed RW$\ell_{1}$ subgradient algorithm, when applied to a random linear system, and taking the method by CWB as reference. For a given level of sparsity $s$, a random linear system $\Phi x=b$ is generated, with a solution $x^{*}$ such that $\|x^{*}\|_{0} \leq s$. The experimental setting is based on \cite{candes_2008}:

\begin{enumerate}
\item Generate $\Phi \in \mathbb{R}^{m \times n}$, with $n=256$, $m=100$ and Gaussian independent entries: $$\Phi_{ij} \sim N \left( 0,\sigma=\frac{1}{\sqrt{m}} \right), \: \forall \: i,j.$$ Note that in particular $\Phi$ will have normalized columns (in expected value).

\item Select randomly a set $I_{s} \subset \{ 1,\ldots,n \}$ of $s$ indexes, representing the coordinates of $x^{*}$ where non-null values are allowed.

\item Generate the values of $x^{*}_{i}, \: i \in I_{s}$, with independent Gaussian distribution: $$x^{*}_{i} \sim N \left( 0, \sigma = \frac{1}{\sqrt{s}} \right), \forall \: i \in I_{s}.$$ Note that in particular $x^{*}$ will be normalized in expected value.

\item Generate the independent term: $b=\Phi x^{*} \in \mathbb{R}^{m}$.
\end{enumerate}

For both RW algorithms, the proposed one and the method by CWB, we use $w^{0}=\vec{1}$. For CWB we take $\epsilon_{k}=0.1, \: \forall \: k \geq 0$. Following \cite{candes_2008}, we say $x^{*}$ was recovered if: \begin{equation}
\|x^{\text{RWIter}}-x^{*}\|_{\infty} \leq 1 \times 10^{-3}.
\end{equation}

For each level of sparsity $s \in [15,55]$, a recovery rate is calculated as the percentaje of recovery over $N_{p}=300$ random problems. Figure \ref{fig:performance_sin_ruido} shows the results for different number of RW iterations. Results for $\ell_{1}$ minimization are also shown for reference. Considering only one RW iteration, the proposed algorithm is slightly better than CWB. This difference disappears for two or more RW iterations, where both algorithms show the same performance; with the additional interpretability of the weights in the proposed methodology.

\begin{figure}
\centering
\includegraphics[width=0.85\textwidth]{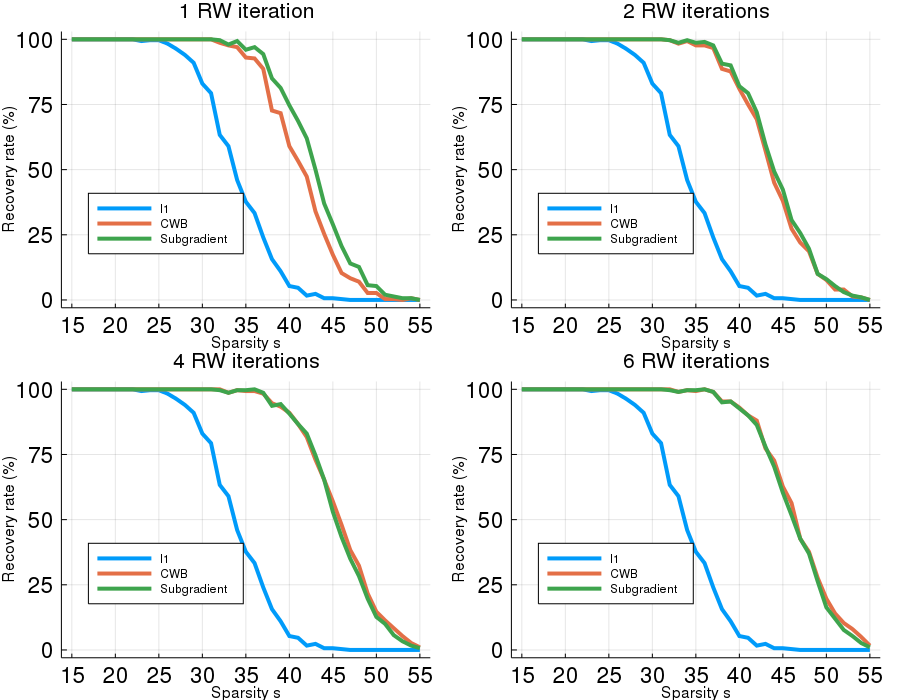}
\caption{Recovery rate of RW$\ell_{1}$ algorithms.} \label{fig:performance_sin_ruido}
\end{figure}

\subsection{Results for the Noisy Setting}

Following \cite{candes_2008}, random problems with noise are generated with $n=256$ and $m=128$. $\Phi$ and $x^{*}$ are generated in the same way as in the noiseless case. Noise $z$ in $b$ is taken with Gaussian independent coordinates, and such that $x^{*}$ is feasible with high probability. For this we take: $z_{i}=\sigma v_{i}, \: v_{i} \sim N(0,1) \text{ independent}$, so: \begin{equation}
\|z\|_{2}^{2} = \sigma^{2} \|v\|_{2}^{2} = \sigma^{2} \left( \sum \limits_{i=1}^{m} v_{i}^{2} \right) \sim \sigma^{2} \chi_{m}^{2}.
\end{equation} Taking for example $\eta^{2} = \sigma^{2} \left( m + 2 \sqrt{2m} \right)$, we have: \begin{equation}
P \left( \|\Phi x^{*}-b\|_{2}^{2} \leq \eta^{2} \right) = 1 - P \left( \chi^{2}_{128} \geq 160 \right) \simeq 0.971.
\end{equation}

We use $w^{0}=\vec{1}$ for both algorithms. For subgradient RW-LASSO we take $\lambda^{0} = \frac{n}{\|z\|_{1}}$, where $z$ is a solution of $\Phi x =b$ with minimum $\ell_{2}$ norm. FISTA algorithm is used for solving each Weighted-LASSO problem \cite{beck_2009}. Performance is measured by the improvement with respect to a solution $x_{\ell_{1}}^{\eta}$ of noisy $\ell_{1}$ minimization: \begin{equation}
a=100 \times \left( 1-\frac{||x^{RW}-x^{*}||_{2}}{||x_{\ell_{1}}^{\eta}-x^{*}||_{2}} \right) \%.
\end{equation}

Figure \ref{fig:performance_ruido} shows the performance with noisy measures for both RW methods: the proposed RW-LASSO algorithm and the RW$\ell_{1}$ CWB algorithm. Results correspond to $N_{p}=300$ tests on random problems with fixed sparsity $s=38$. The mean improvement $\bar{x}$ is also shown (vertical red line), together with $\pm$ one standard deviation $\bar{\sigma}$ (vertical violet and green lines). CWB RW$\ell_{1}$ algorithm shows a mean improvement of $21\%$ with respect to $\ell_{1}$ minimization. For RW-LASSO subgradient this improvement is $32\%$, significantly higher than CWB.

\begin{figure}
\centering
\includegraphics[width=0.75\textwidth]{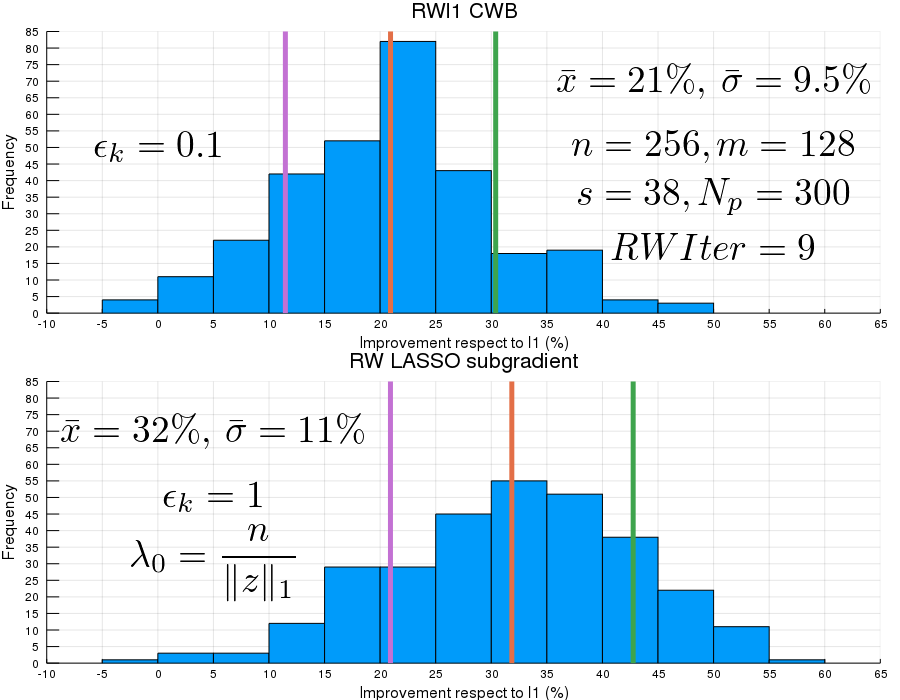}
\caption{Performance of RW algorithms with respect to $\ell_{1}$ minimization (with noise).} \label{fig:performance_ruido}
\end{figure}

We also considered the RW-LASSO algorithm with weights updated as CWB, but the performance was very poor. The reason for this may be that $\lambda_{k}$ remains fixed at $\lambda_{0}$, as there is no obvious rule for updating it.

\section{Conclusions} \label{sec:conclusiones}

In this paper the important problem of finding sparse solutions of a linear system was considered. A usual alternative to this NP-hard problem is the Weighted $\ell_{1}$ problem, where the choice of weights is crucial. A new methodology for estimating weights was proposed, based on identifying weights as solutions of a Lagrange dual problem. It was shown that this problem always admits ``useful'' solutions. The proposed methodology was then applied using the projected subgradient algorithm, obtaining a RW$\ell_{1}$ algorithm, alternative to the classical one, due to CWB. This new algorithm was tested on random problems in the noiseless case, obtaining almost the same performance as that of CWB, but allowing an interpretation of weights. The proposed methodology was then extended to the noisy case. Here a RW-LASSO algorithm was obtained, by introducing a new Lagrange multiplier. This last algorithm showed a considerable improvement in performance, with respect to the RW$\ell_{1}$ algorithm proposed by CWB.

\subsubsection*{Acknowledgment.}

This work was supported by a grant from Comisi\'on Acad\'emica de Posgrado (CAP), Universidad de la Rep\'ublica, Uruguay.

% ---- Bibliography ----
% BibTeX users should specify bibliography style 'splncs04'.
% References will then be sorted and formatted in the correct style.
\bibliographystyle{splncs04}
\bibliography{bibliografia/biblio}

%
%\begin{thebibliography}{8}
%
%\bibitem{ref_article1}
%Author, F.: Article title. Journal \textbf{2}(5), 99--110 (2016)
%
%\bibitem{ref_lncs1}
%Author, F., Author, S.: Title of a proceedings paper. In: Editor,
%F., Editor, S. (eds.) CONFERENCE 2016, LNCS, vol. 9999, pp. 1--13.
%Springer, Heidelberg (2016). \doi{10.10007/1234567890}
%
%\bibitem{ref_book1}
%Author, F., Author, S., Author, T.: Book title. 2nd edn. Publisher,
%Location (1999)
%
%\bibitem{ref_proc1}
%Author, A.-B.: Contribution title. In: 9th International Proceedings
%on Proceedings, pp. 1--2. Publisher, Location (2010)
%
%\bibitem{ref_url1}
%LNCS Homepage, \url{http://www.springer.com/lncs}. Last accessed 4
%Oct 2017
%
%\end{thebibliography}

\end{document}